\documentclass[10pt]{article}
\usepackage[utf8]{inputenc}

\usepackage{hyperref}
\usepackage{xcolor}
\usepackage{graphics}
\hypersetup{
  colorlinks,
  linkcolor={blue!80!black},
  urlcolor={blue!80!black},
  citecolor={blue!80!black}
}

\usepackage{amsmath,amssymb,amsthm}
\usepackage{mathrsfs}
\usepackage{graphics}
\usepackage{amsmath}
\usepackage{wrapfig}

\numberwithin{equation}{section}

\usepackage{authblk}

\usepackage{cite} 

\newtheorem{theorem}{Theorem}[section]
\newtheorem{lemma}[theorem]{Lemma}

\newtheorem{proposition}[theorem]{Proposition}

\theoremstyle{definition}
\newtheorem{definition}[theorem]{Definition}

\newcommand{\R}{{\mathbb R}}
\newcommand{\N}{{\mathbb N}}
\newcommand{\C}{{\mathbb C}}
\newcommand{\abs}[1]{{\left\lvert{#1}\right\rvert}}
\newcommand{\norm}[1]{{\left\lVert{#1}\right\rVert}}

\title{Unique determination of the shape of a scattering screen from a
  passive measurement}

\author[1,2]{Emilia Bl{\aa}sten}
\author[1]{Lassi P\"aiv\"arinta}
\author[1]{Sadia Sadique}

\affil[1]{{\small Division of Mathematics, Tallinn University of
    Technology, Department of Cybernetics, 19086 Tallinn, Estonia}}

\affil[2]{{\small Department of Mathematics and Statistics, University
    of Helsinki, 00014 Helsinki, Finland}}

\begin{document}
\maketitle

\abstract{We consider the problem of fixed frequency acoustic
  scattering from a sound-soft flat screen. More precisely the
  obstacle is restricted to a two-dimensional plane and interacting
  with a arbitrary incident wave, it scatters acoustic waves to
  three-dimensional space. The model is particularly relevant in the
  study and design of reflecting sonars and antennas, cases where one
  cannot assume that the incident wave is a plane wave. Our main
  result is that given the plane where the screen is located, the
  far-field pattern produced by any single arbitrary incident wave
  determines the exact shape of the screen, as long as it is not
  antisymmetric with respect to the plane. This holds even for screens
  whose shape is an arbitrary simply connected smooth domain. This is
  in contrast to earlier work where the incident wave had to be a
  plane wave, or more recent work where only polygonal scatterers are
  determined.\\

  \medskip\flushleft\textbf{keywords}: inverse scattering; screen;
  uniqueness; single measurement; passive measurement\\
  
}

\section{Introduction}

\subsection{Antennas}

The motivation for the study of wave scattering from thin and large
objects lies in the antenna theory. The starting point for this was
when the Prussian Academy announced an open competition about who
could be the first to show the existence or non-existence of
electromagnetic (EM) waves in 1879. The existence of these waves were
predicted fifteen years earlier by the mathematical theory of James
Clerk Maxwell \cite{Maxwell}. The competition was won in 1882 by young
Heinrich Hertz, in favour of Maxwell's theory. He did this by
constructing a dipole antenna radiating EM waves which he could
measure. It is needless to mention the importance which this
experiment together with Maxwell's theory has had for modern
society. Hertz’s antenna consisted of two identical perfectly
conducting planar bodies, in his case squares, which create radiating
EM waves. Since, by reciprocity, radiating antennas are identical to
receiving antennas, the theory of antennas is closely connected to EM
scattering and inverse scattering theory.

A key question in antenna design for scientific radio arrays is how to
choose the antenna topology so that its impedance and radiation
pattern are frequency independent (FI) over a wide range of
frequencies and, simultaneously, the radiation pattern supports
beamforming. Well-known examples of FI antennas include log-periodic,
log-spiral, and UHF fractal antennas on high-frequencies. While proven
good for extremely wide band work, these are heavy and complicated
structures and thus not cost-efficient for extremely large arrays.

Instead of relying on traditional antenna forms, we aim to derive
general principles for designing antennas with frequency independent
characteristics. A major step in such a design strategy is to solve
the inverse scattering problem: given an input--output pair of waves,
which antenna shape produces it? The input is a given incident wave,
and the output is the far-field pattern produced by the antenna. The
path to antenna design is a long one, so in this paper we study the
technically easier accoustic scattering problem.

In acoustics, scattering surfaces or screens are not called antennas
but sonars. Traditionally sonars are classified into active and
passive sonars, depending of whether they act as a sound source or
receiver. We consider acoustic scattering from screens, something
which lies between these two extremes. It is more correct to call
these screens passive sonars as they do not have an energy source,
however they are active in the sense that their effect on the sound
pattern is significant. In general the nomenclature ``sonar'' refers
to probing using an active and passive sonar. Our research is rather
in the domain of acoustic design. The mathematical question of finding
a screen that scatters a given incident wave into a particular
far-field has applications like the following, for example: how to
reduce echo in an office space? How to direct acoustic vibrations or
reduce them?  Of course, it also answers the probing question: can we
determine the shape and location of a passive sonar by how it reflects
sound? These are complex questions, only one part of which we are
going to solve, namely that a single input--output pair of sound waves
uniquely determines the shape of a flat acoustic screen.

\subsection{Mathematical background}

The problem of inverse scattering with reduced measurement data has
gained a lot of interest lately. Traditionally determining a scatterer
from far-field measurements requires sending all possible incident
waves and recording the corresponding far-field patterns. The method
of using \emph{complex geometrical optics solutions} and infinitely
many far-field measurements in the fixed frequency setting was
pioneered by Sylvester and Uhlmann in \cite{SU}, and was the first
method for uniquely determining an arbitrary smooth enough scattering
potential by far-field measurements. The field has grown extremely
fast since then, almost to the point of saturation, and we will only
point the reader towards the surveys in \cite{Uhl} for references up
to 2003, which gives a good picture of the situation except for
scattering in two dimensions, which was solved by Bukhgeim
\cite{Bukhgeim} in 2007 and improved by sevaral authors,
e.g. \cite{GT,IUY,DSFKS,BIY,BTW}.

In many applications the scatterer is impenetrable, or we are only
interested in its shape or location. The shape determination problem
is known as Schiffer's problem in the literature \cite{CK}.
M.~Schiffer showed that a sound-soft obstacle (with non-empty
interior) can be uniquely determined by infinitely many far-field
patterns. The proof appeared as a private communication in the
monograph by Lax and Phillips \cite{LP}. Linear sampling
\cite{CKsampling} and factorization \cite{KG} methods were developped
and they are very well suited for shape determination, also from the
numerical point of view. These were applied in the context of curved
screens in acoustic \cite{AH} and electromagnetic \cite{CCD}
scattering to determine the shape and location of the screen, also
numerically. However these methods require the full use of infinitely
many far-field patterns, except for a case of interest in \cite{AH} to
which we will return later on in more detail.

There was still much to improve: counting dimensions shows that a
single far-field (a mapping $\mathbb S^{n-1}\to\C$) should be enough
to determine the shape (a manifold of dimension $n-1$). Colton and
Sleeman reduced the requirements to finitely many far-field patterns
\cite{CS}. It is widely conjectured that the uniqueness for Schiffer's
problem follows from a single far-field pattern \cite{CK,Isa2}, and
the situation for a general shape is wide open. This brings in the
current results. Various authors proved at roughly the same time in
the recent past that polyhedral sound-soft obstacles are uniquely
determined by a single far-field pattern in various settings
\cite{AR,CY,EY2,LPRX,LRX,Liu-Zou,Ron2}. Part of the results above
apply for screens as long as the screen is polygonal. A special case
in \cite{AH} gives the unique determination of a flat screen by a
single incident plane-wave measurement. Their proof requires that the
incident wave has non-vanishing properties everywhere on the plane
where the screen is located --- an issue that we remedy completely. So
far there is no proof for the unique determination of an obstacle's
shape by one far-field pattern without restrictive a priori
assumptions. The results in \cite{HNS} come very close: the obstacle
can be any Lipschitz domain as long as its boundary is not an analytic
manifold. It does not allow screens, which is our focus.

An alternative approach to unique determination which has gained
interest recently, is to consider what can be determined with less
data, e.g. one measurement, in the setting of penetrable scatterers
which were usually treated with various methods based on the
Sylvester--Uhlmann \cite{SU} or Bukhgeim \cite{Bukhgeim} papers. Much
of the recent work taking this point of view uses unique continuation
results and precise analysis on the behaviour of Fourier transforms of
the characteristic functions of various shapes
\cite{BPS,PSV,Bsource,BL2016,BL2017,HSV,Ikehata}. A very interesting
point of view is determining the so-called convex scattering support
\cite{KS1,KS2} by one far-field measurement. Again, none of the above
are applicable to screens per se.

\smallskip
Our work in this paper shows that given the far-field caused by any
single given incident wave scattering off of a smooth flat screen, the
latter's shape is determined uniquely. Our methods are based on ideas
which are partly motivated by the study of certain integral operators
in \cite{PR0,PR}. As in \cite{AH}, we first show that the far-field is
the restriction to a ball of radius $k$ (the wavenumber) of the
two-dimensional Fourier transform of a function supported on the
screen. Next, since the incident wave might vanish on part of the
screen, we show that the shape of the screen is exactly the support of
that function. This latter part involves a delicate analysis of the
Taylor coefficients of the scattered wave at the screen, but it leads
to our main theorem: that Schiffer's problem is uniquely solvable for
flat screens on a plane in three dimensions, for any incident wave
that causes scattering.

Let us discuss the significance of our result, with focus especially
on our improvements over \cite{AH}: that any incident field is
allowed. We will start with the mathematical challenges. Unlike for
infinite measurements inverse problems such as \cite{SU,Bukhgeim},
properties of the incident wave affect greatly the solvability of
single measurement inverse problems. Complex plane waves make things
technically simpler in many scattering problems because of their
explicit form and non-vanishing everywhere. This often reduces the
non-linear inverse scattering problem to the linear inverse source
problem after a suitable interpretation, or avoids other challenges,
as can be seen by comparing \cite{BPS,PSV,HSV} to
\cite{Bsource,BL2016,BL2017}. Futhermore, in situations involving
scattering from multiple objects, the total incident field impinging
on a given component is the sum of the original incident field and the
fields scattered by the other components. This is relevant when one
wishes to uniquely determine a screen where space contains other
scatterers that are known. On the other hand, from the applied point
of view, solving the inverse problems for any given incident field
enables \emph{passive measurements}. This means that even if we do not
have control over the incident wave, or cannot afford to control it,
the shape of the scatterer can be uniquely determined. This is both
good and bad. It means that the flat screen design problem of finding
its shape such that it scatters one given incident wave into a given
far-field has no more than a unique solution. On the other hand it
shows the impossibility of more complex input--output systems. One
cannot require it to scatter two or more incident waves into their
corresponding far-fields in general. The first incident-wave and
far-field pair already determines the shape.

Lastly, we remark that inverse scattering for screens has still many
open problems. Current solutions require that the screen have at least
a differentiable boundary, something which arises from the way that
the direct scattering problem has been shown solvable in
\cite{Stephan} and other sources. To bring forward the range
characterization condition from \cite{AH} to the situations of let's
say Herglotz incident waves, one would need to solve a deconvolution
problem. A more difficult and certainly more interesting question
mathematically and from the point of view of applications, is the
unique determination of the shape of a curved screen from one
measurement, passive or fully controlled. The problem is solved for
infinitely many measurements in \cite{AH}, but counting dimensions
suggests that it should be solvable with one measurement.

\subsection{Definitions and Theorems}
Let us go forward to the mathematics. We start by defining what we
mean by a screen and the scattering problem from screens. Then we
state our three main theorems. They give representation formulas for
the scattered wave, the far-field pattern, and the unique solvability
of Schiffer's problem for determining the shape of a scattering screen
using a single incident wave. In Section~\ref{sect:repThms} we prove
the representation formulas, and then in Section~\ref{sect:inverse} we
solve the inverse problem.

\medskip
We consider the scattering of a two dimensional sound-soft and flat
obstacle $\Omega$ in three dimensional space. We will assume that
$\Omega$ is an open subset of $\R^2\times\{0\}$.

\begin{definition}
  We call a set $\Omega\subset\R^3$ \emph{a screen}, if $\Omega =
  \Omega_0 \times \{0\}$ for some simply connected bounded domain
  $\Omega_0\subset\R^2$ whose boundary is smooth, and which we call
  its \emph{shape}.
\end{definition}

The scattering of acoustic waves by $\Omega$ leads to the study of the
Helmholtz equation $(\Delta +k^2)u=0$ where the wave number $k$ is
given by the positive constant $k=\omega/c$ where $c$ is the constant
speed of sound in the background fluid (air, water, etc) and $\omega$
is the angular frequency of the wave. The pressure of the total wave
vanishes on the boundary of a sound-soft obstacle, and the total wave
is a sum of the incident and scattered waves. This leads to the
following set of partial differential equations.

\begin{definition} \label{directScat}
  We define the \emph{direct scattering problem for a screen $\Omega$}
  as follows. Given an incident wave $u_i$ satisfying
  $(\Delta+k^2)u_i=0$ in $\R^3$ and a screen $\Omega$, the direct
  scattering problem has a solution if there is $u_s\in
  H^1_{loc}(\R^3\setminus\overline\Omega)$ that satisfies the
  following conditions
  \begin{align}
    &(\Delta +k^2) u_s = 0, & \R^3\setminus\overline\Omega,\label{helmholtz}\\
    &u_i(x) + u_s(x) = 0, &x\in\Omega,\label{dirichletCondition}\\
    &r \Big(\frac{\partial}{\partial r} - ik\Big)u_s = 0, &r\to\infty, \label{sommerfeld}
  \end{align}
  where $r=\abs{x}$ and the limit is uniform over all directions $\hat
  x = x/r \in \mathbb S^2$ as $r\to\infty$.
\end{definition}
There are a few things above that we should clarify. By
$H^1_{loc}(\R^3\setminus\overline\Omega)$ we mean the set of
distributions $\psi$ on $\R^3\setminus\overline\Omega$ for which
$\psi_{|U} \in H^1(U)$ for any bounded convex open set
$U\subset\R^3\setminus\overline\Omega$. Secondly, since strictly
speaking $u_s$ is not defined on $\Omega$, by
\eqref{dirichletCondition} we mean that the Sovolev trace of $u_s$
both from above ($x_3>0$) and below ($x_3<0$) coincides, and is equal
to $-u_i$ on $\Omega$.

We shall start by showing a representation formula \eqref{usRho} for
solutions $u_s$ of the direct scattering problem for the screen. This
is mainly done so that the reader would get a better intuition about
this type of problems and to fix notation and function spaces clearly.
This formula is well known, and it gives a unique solution to the
direct problem \cite{Stephan}. After that we will show that the
far-field, defined below, corresponding to a single given non-trivial
incident wave uniquely determines the screen $\Omega$. This type of
theorem was shown in \cite{AH} on the condition that the incident wave
does not vanish on the plane $\R^2\times\{0\}$. To get rid of this
assumption, we have to show Lemma~\ref{OmegaDetermination}. We remark
that the far-field pattern exists and is unique for each $u_s$
satisfying the following assumptions. See \cite{CK} for reference.
\begin{definition}
  Let $u_s$ satisfy the Sommerfeld radiation condition of
  \eqref{sommerfeld} and the Helmholtz equation $(\Delta+k^2)u_s=0$
  outside a ball $B\subset\R^3$. We say that $u_s^\infty$ is the
  far-field of $u_s$ if
  \[
  u_s(x) = \frac{e^{ik\abs{x}}}{\abs{x}} \left( u_s^\infty(\hat x) +
  \mathcal O\left(\frac{1}{\abs{x}}\right) \right)
  \]
  uniformly over $\hat x$ as $x\to\infty$.
\end{definition}

We define some notation which will be useful throughout the whole
text.
\begin{itemize}
\item $x,y,\ldots$ represent variables in $\R^3$, and we associate to
  them various projections described below.
\item $x',y',\ldots$ mean variables in $\R^2$ or projections to
  $\R^2$. For example if $x = (1,2,3)\in\R^3$ then in that context
  $x'=(1,2)\in\R^2$, but we could have $dy'$ in an integral over a
  subset of $\R^2$ without having to define the variable $y$
  separately.
\item $x^0,y^0,\ldots$ denote lifts to $\R^3$, meaning $x^0 =
  (x',0)$. For example if $x' = (-1,-2)$ then $x^0=(-1,-2,0)$. This
  notation can also be used as a projection
  $\R^3\to\R^2\times\{0\}$. So if $x=(1,2,3)$ then $x^0 =
  (1,2,0)$. Essentially ${x'}^0 = (x')^0 = x^0$ and ${x^0}' = (x^0)' =
  x'$ but we do not use this combined notation explicitly.
\item $\Phi$ is reserved for the fundamental solution to
  $(\Delta+k^2)$, defined in Lemma~\ref{fundamentalSolution}.
\item $u^+,u^-$ mean the function $u$ restricted to $\R^2\times\R_+$
  and $\R^2\times\R_-$, respectively. If their variable is in
  $\R^2\times\{0\}$ then they are the two-sided limits (traces) as
  $x_3\to0$. We often use $\partial_3 u^+$ and $\partial_3 u^-$. These
  are simply the derivatives in the $x_3$-direction of $u^+$ and
  $u^-$, respectively. Often this is evaluated on $\R^2\times\{0\}$
  where it then denotes the one-sided derivative, i.e. the trace of
  $\partial_3 u^\pm$.
\item $\widetilde H^{-1/2}(\Omega_0)$: this is the set of
  $H^{-1/2}(\R^2)$ distributions whose support is contained in
  $\overline{\Omega_0}$, where we recall that $\Omega_0$ signifies the
  shape of a screen $\Omega$.
\end{itemize}

\bigskip
Let us discuss the direct scattering problem
\eqref{helmholtz}--\eqref{sommerfeld} first. In
Section~\ref{sect:repThms}, Proposition~\ref{prop:usRep}, we will show
the well-known representation formula
\begin{equation} \label{usRho}
  u_s(x) = \int_{\R^2} \Phi(x,y^0) \rho(y') dy'
\end{equation}
for all $x\in\R^3\setminus\overline\Omega$, where
\begin{equation} \label{rhoDef}
  \rho(y') = \partial_3 u_s^+(y^0) - \partial_3 u_s^-(y^0)
\end{equation}
is an element of $\widetilde H^{-1/2}(\Omega_0)$ and the integral in
\eqref{usRho} is interpreted as a distribution pairing between $\rho$
and the smooth test function $\Phi$ restricted to the screen. Taking
the trace $x\to\Omega$ in \eqref{usRho} and recalling that $u_s = -
u_i$ on $\Omega$ in the sense of traces, \eqref{dirichletCondition},
we get
\begin{equation} \label{intEq}
  u_i(x) = - \int_{\R^2} \Phi(x,y^0) \rho(y') dy'.
\end{equation}
Now, for any candidate solution $u_s\in
H^1_{loc}(\R^3\setminus\overline\Omega)$, it solves the direct problem
\eqref{helmholtz}--\eqref{sommerfeld} if and only if $\rho$, as
defined above, is in $\widetilde H^{-1/2}(\Omega_0)$ and is the
solution to \eqref{intEq}. More precisely, given $\rho$ solving the
integral equation, we can define $u_s$ by \eqref{usRho}, and it would
solve the direct scattering problem. This was shown in Theorem~2.5 in
\cite{Stephan}. Theorem~2.7 in the same source proves that
\eqref{intEq} has a unique solution $\rho\in \widetilde
H^{-1/2}(\Omega_0)$ given any $u_i \in H^{1/2}(\Omega_0)$.

\medskip
Our main contributions are the following. The first of which is the
familiar far-field representation derived from \eqref{usRho} if $\rho$
is a function. We generalize is to distributions in
$H^{-1/2}(\R^2)$. This is required for consistency of the function
spaces involved. This detail has not been stated explicitely in
earlier work involving scattering from screens.
\begin{theorem} \label{usFarField}
  Let $\Omega\subset\R^3$ be a screen and $u_s$ satisfy the direct
  scattering problem for some incident field $u_i$ and screen
  $\Omega$. Then its far-field has the representation
  \begin{equation}
    u^\infty_s(\hat x) = \frac{1}{4\pi} \left\langle
    \big(\partial_3u_s^+ - \partial_3u^-\big)(y^0), e^{-ik\hat x\cdot
      y^0} \right\rangle_{y'}
  \end{equation}
  for $\hat x \in \mathbb S^2$. If $\partial_3 u_s^+ - \partial_3
  u_s^-$ is integrable on $\Omega$, this formula is equivalent to
  \[
  u^\infty_s(\hat x) = \frac{1}{4\pi} \int_{\R^2} e^{-ik\hat x\cdot
    y^0} \big(\partial_3u_s^+ - \partial_3u^-\big)(y^0) dy'.
  \]
\end{theorem}

Our main theorem shows that even with an unoptimal incident wave, the
scattering caused by it from flat screens determines the shape
uniquely.
\begin{theorem} \label{inverseSolution}
  Let $\Omega,\tilde\Omega\subset\R^3$ be screens and $k\in\R_+$. Let
  $u_i$ be an incident wave and $u_s, \tilde u_s$ be scattered waves
  that satisfy the direct scattering problem for screens
  $\Omega,\tilde\Omega$, respectively.

  If $u_i$ is not antisymmetric with respect to $\R^2\times\{0\}$ and
  $u_s^\infty = \tilde u_s^\infty$, then $\Omega = \tilde\Omega$. If
  it is antisymmetric then $u_s^\infty = \tilde u_s^\infty = 0$ for
  any screens $\Omega,\tilde\Omega$.
\end{theorem}

\section{Representation theorems} \label{sect:repThms}

In this section we will prove that solutions to the direct scattering
problem satisfy \eqref{usRho}. In essence we present the well-known
but very condensed argument of \cite{Stephan} in more detail for the
convenience of the readers. We will start with representation formulas
for smooth functions and then approximate the $H^1$-smooth $u_s$. At
the end of the section we will prove Theorem~\ref{usFarField}.

\begin{lemma} \label{C2green}
  Let $D\subset\R^3$ be a bounded domain whose boundary is piecewise
  of class $C^{1}$ and let $\nu$ denote the unit normal vector to the
  boundary $\partial D$ directed to the exterior of $D$. Then, for $u,
  v \in C^2(\overline{D})$ we have Green's second formula
  \begin{equation} \label{eq 5}
    \int_D (v \Delta u - u \Delta v) dx = \int_{\partial D} \Big(
    \frac{\partial u}{\partial \nu } v - u \frac{\partial v}{\partial
      \nu } \Big) ds
  \end{equation}
  where $ds$ is the surface measure of $\partial D$.
\end{lemma}
\begin{proof}
  Theorem~3 in Appendix~C.2 of \cite{evans}.
\end{proof}

\begin{lemma} \label{fundamentalSolution}
  Let $D\subset\R^3$ be a bounded domain whose boundary is piecewise
  of class $C^1$ and $k\in\R_+$. Let
  \[
  \Phi(x,y) = \frac{e^{ik\abs{x-y}}}{4\pi\abs{x-y}}
  \]
  for $x,y\in\R^3$, $x\neq y$. Then for any $\varphi\in C^2(\overline
  D)$ and $x\in\R^3\setminus\partial D$ we have
  \begin{align}
    \int_D \Phi(x,y) (\Delta+k^2)\varphi(y) dy &= \int_{\partial D}
    \big( \Phi(x,y) \partial_\nu \varphi(y) - \varphi(y) \partial_\nu
    \Phi(x,y)\big) ds(y) \notag\\ & \quad+
    \begin{cases} 0,
      &x\in\R^3\setminus\overline{D},\\ -\varphi(x), &x\in D.
    \end{cases}
    \label{greenRepresentation}
  \end{align}
\end{lemma}
\begin{proof}
  We have $(\Delta+k^2)\varphi$ bounded and $y\mapsto \Phi(x,y)$
  integrable for any $x$, so
  \[
  \int_D \Phi(x,y) (\Delta+k^2)\varphi(y) dy = \lim_{r\to0}
  \int_{D\setminus B(x,r)} \Phi(x,y) (\Delta+k^2)\varphi(y) dy.
  \]
  Green's second formula \eqref{eq 5} applied to the integral on the
  right gives
  \begin{align*}
    \ldots &= \int_{D\setminus B(x,r)} (\Delta+k^2)\Phi(x,y)
    \varphi(y) dy \\ &\phantom{=}+ \int_{S(x,r)\cap\overline D} \big(
    \Phi(x,y) \partial_\nu \varphi(y) - \varphi(y) \partial_\nu
    \Phi(x,y)\big) ds(y) \\ &\phantom{=}+ \int_{\partial
      D\setminus\overline B(x,r)} \big( \Phi(x,y) \partial_\nu
    \varphi(y) - \varphi(y) \partial_\nu \Phi(x,y)\big) ds(y).
  \end{align*}
  The first integral here vanishes because $(\Delta_y+k^2)\Phi(x,y)=0$
  when $y\neq x$.

  The integral over $\partial D\setminus\overline B(x,r)$ gives the
  second term in the claim when $r\to0$ because $\Phi, \partial\Phi$
  are integrable since $x\notin\partial D$. Let us estimate the first
  term in the first boundary integral. We have
  \[
  \int_{S(x,r)\cap\overline D} \Phi(x,y) \partial_\nu \varphi(y) ds(y)
  = \int_{S(x,r)\cap\overline D} \frac{e^{ikr}}{4\pi r}
  \partial_\nu(y) ds(y)
  \]
  and by the ML-inequality we have
  \[
  \abs{ \int_{S(x,r)\cap\overline D} \Phi(x,y) \partial_\nu \varphi(y)
    ds(y) } \leq \frac{1}{4\pi r} \sup_{y\in S(x,r) \cap \overline D}
  \abs{\nabla \varphi(y)} 4\pi r^2 \to 0
  \]
  as $r\to0$ because $\abs{\nabla \varphi}$ has a uniform bound in
  $\overline D$. In the last integral we have $\partial_nu \Phi(x,y) =
  -\partial_r \big(e^{ikr}/(4\pi r)\big) = -ik e^{ikr} / (4\pi r) +
  e^{ikr}/(4\pi r^2)$. The integral involving $ik e^{ikr}/(4\pi r)$
  can be estimated as above to conclude that it vanishes when
  $r\to0$. The remaining integral is
  \begin{align*}
    &- \frac{e^{ikr}}{4\pi r^2} \int_{S(x,r) \cap\overline D}
    \varphi(y) ds(y) \\&\qquad = - \frac{e^{ikr}}{4\pi r^2}
    \int_{S(x,r) \cap\overline D} \big(\varphi(y) - \varphi(x)\big)
    ds(y) - \frac{e^{ikr}}{4\pi r^2} \varphi(x)
    s\big(S(x,r)\cap\overline D\big).
  \end{align*}
  We have $\abs{\varphi(y)-\varphi(x)} \leq \sup_{\xi\in\overline D}
  \abs{\nabla \varphi(\xi)} \abs{x-y}$ so the absolute value of the
  first integral above can be estimated as
  \[
  \ldots \leq \frac{\sup \abs{\nabla \varphi}}{4\pi r^2}
  \int_{S(x,r)\cap\overline D} \abs{x-y} dy =
  \frac{\sup\abs{\nabla\varphi}}{4\pi r^2} r s\big(S(x,r)\cap\overline
  D\big) \to 0
  \]
  as $r\to0$. The form of the remaining term implies the claim in each
  of the cases $x\in D$, $x\in \R^3\setminus\overline D$.
\end{proof}

\begin{lemma} \label{H1integrationByParts}
  Let $D\subset\R^3$ be a bounded domain with smooth boundary and
  $k\in\R_+$. Let $u_s \in H^1(D)$ with $(\Delta+k^2)u_s \in
  L^2(D)$. Then
  \begin{align}
    u_s(x) &= -\int_D \Phi(x,y)(\Delta+k^2)u_s(y) dy \notag \\ &\quad
    + \int_{\partial D} \big( \Phi(x,y)\partial_\nu u_s(y) - u_s(y)
    \partial_\nu\Phi(x,y) \big) ds(y) \label{usDistributionInside}
  \end{align}
  for $x\in D$ in the distribution sense. For
  $x\in\R^3\setminus\overline D$ we have
  \begin{align}
    0 &= -\int_D \Phi(x,y)(\Delta+k^2)u_s(y) dy \notag \\ &\qquad +
    \int_{\partial D} \big( \Phi(x,y)\partial_\nu u_s(y) - u_s(y)
    \partial_\nu\Phi(x,y) \big) ds(y) \label{usDistributionOutside}
  \end{align}
  in the distribution sense. Here the boundary integrals involving
  $\partial_\nu u_s$ are to be interpreted as distribution pairings
  between a $H^{-1/2}(\partial D)$ function and a test function.
\end{lemma}
\begin{proof}
  We will prove only the first case, namely $x\in D$. The second one
  follows similarly. Let $(\varphi_j)_{j=0}^\infty$ be a sequence of
  smooth functions defined on $\overline D$ such that
  \[
  \norm{u_s - \varphi_j}_{H^1(D)} +
  \norm{(\Delta+k^2)(u_s-\varphi_j)}_{L^2(D)} \to 0
  \]
  as $j\to\infty$. Such a sequence exists, for example by convolving
  $u_s$ with a mollifier $\psi_\varepsilon$, as in $\varphi_j = (u_s
  \ast \psi_{1/j})_{|\overline{D}}$.

  We have $\Phi(x,y) = \Psi(x-y)$ for $\Psi(z) =
  \exp(ik\abs{z})/(4\pi\abs{z})$ which is locally integrable in
  $\R^3$. Hence the first term in the right-hand side of
  \eqref{usDistributionInside}, equal to $\Psi \ast (\Delta+k^2)u_s$,
  can be approximated by $\Psi\ast(\Delta+k^2)\varphi_j$ in the
  $L^2(D)$-sense.

  For any $x\in D$ the second integral in \eqref{usDistributionInside}
  is well defined because $y\mapsto\Phi(x,y)$ and
  $y\mapsto\partial_\nu\Phi(x,y)$ are smooth on the smooth manifold
  $\partial D$. Moreover the $x$-dependence is smooth, so the mapping
  \[
  u_s \mapsto \int_{\partial D} u_s(y) \partial_\nu\Phi(x,y) ds(y)
  \]
  is bounded $H^1(D) \to H^{1/2}(\partial D) \to C^0(D)$ and similarly
  \[
  u_s \mapsto \int_{\partial D} \Phi(x,y)\partial_\nu u_s(y)\big)
  ds(y)
  \]
  is bounded $H^1(D) \to H^{-1/2}(\partial D) \to C^0(D)$ when the
  integral is interpreted as a distrubion pairing between a
  $H^{-1/2}(\partial D)$-function and a test function. The continuity
  does not necessarily hold up to the boundary. Because $\varphi_j \to
  u_s$ in $H^1(D)$ and the trace operators map $\operatorname{Tr} :
  H^1(D) \to H^{1/2}/D)$, $\partial_\nu : H^1(D) \to H^{-1/2}(\partial
  D)$, so the boundary integrals with $u_s$ replaced by $\varphi_j$
  converge to the corresponding ones in $C^0(D)$, namely uniformly
  over compact subsets of $D$.

  In conclusion, for a test function $\psi\in C^\infty_0(D)$ we have
  \begin{align*}
    \langle u_s, \psi\rangle &= \lim_{j\to\infty} \langle \varphi_j,
    \psi\rangle \\ &= \lim_{j\to\infty} \Bigg\langle -\int_D
    \Phi(x,y)(\Delta+k^2)\varphi_j(y) dy \\ &\phantom{=
      \lim_{j\to\infty}\Bigg\langle} + \int_{\partial D} \big(
    \Phi(x,y)\partial_\nu \varphi_j(y) - \varphi_j(y)
    \partial_\nu\Phi(x,y) \big) ds(y), \psi(x) \Bigg\rangle_x \\ &=
    \Bigg\langle -\int_D \Phi(x,y)(\Delta+k^2)u_s(y) dy
    \\ &\phantom{=\Bigg\langle} + \int_{\partial D} \big(
    \Phi(x,y)\partial_\nu u_s(y) - u_s(y) \partial_\nu\Phi(x,y) \big)
    ds(y), \psi(x) \Bigg\rangle_x \\
  \end{align*}
  so the equality holds in $\mathscr D'(D)$.
\end{proof}

\begin{proposition} \label{prop:usRep}
  Let $\Omega\subset\R^3$ be a screen, $k\in\R_+$ and $\Phi$ the
  fundamental solution from Lemma~\ref{fundamentalSolution}. Let $u_s
  \in H^1_{loc}(\R^3\setminus\overline\Omega)$. If $(\Delta+k^2)u_s=0$
  in $\R^3\setminus\overline\Omega$ and it satisfies the Sommerfeld
  radiation condition, then
  \begin{equation} \label{usRepresentation}
    u_s(x) = \int_{\R^2} \Phi(x,y^0) (\partial_3 u_s^+ - \partial_3
    u_s^-)(y^0) dy'
  \end{equation}
  for $x\in\R^3\setminus\overline\Omega$. Also $y'\mapsto (\partial_3
  u_s^+ - \partial_3 u_s^-)(y^0)$ is in $\widetilde
  H^{-1/2}(\Omega_0)$, and more precisely the integral above
  represents the distribution pairing of a $\widetilde
  H^{-1/2}(\Omega_0)$-function with the smooth test function $\Phi$
  restricted to $\R^2\times\{0\}$ on the $y$-variable.
\end{proposition}
\begin{proof}
  Fix $x\in\R^3\setminus\overline\Omega$. Let $D\subset\R^3$ be a
  bounded domain with smooth boundary for which $x\in D$ and
  $\Omega\subset\partial D$ and furthermore we want this set to be
  \emph{on top} of $\Omega$, namely that its boundary normal pointing
  to the interior at $\Omega$ is $e_3$ and not $-e_3$. Let $R >
  \sup_{z\in D} \abs{x-z}$. We will use the formulas of
  Lemma~\ref{H1integrationByParts} on $D$, which has $\Omega$ on its
  boundary, and $B(x,R)\setminus\overline D$.

  Firstly note that since $(\Delta+k^2)u_s=0$ only the boundary
  integrals on the right-hand sides of \eqref{usDistributionInside}
  and \eqref{usDistributionOutside} remain. We will see the first
  integral as is, namely
  \begin{equation} \label{usDin}
    u_s(x) = \int_{\partial D} \big( \Phi(x,y)\partial_\nu^D u_s(y) -
    u_s(y) \partial_\nu^D \Phi(x,y) \big) ds(y),
  \end{equation}
  where we denote by $\partial_\nu^D$ the internal boundary normal
  derivative of $D$, applied to functions on $D$. We will have the
  integrals in \eqref{usDistributionOutside} to be over the set
  $B(x,R)\setminus\overline D$. The boundary of this set is $S(x,r)
  \cup \partial D$, and the boundary normal pointing to its interior
  is $-e_3$ on $\Omega \subset \partial (B(x,R)\setminus\overline
  D)$. We will split the boundary integral accordingly, and in the
  integral over $\partial D$ we denote by $\partial_\nu^{D_c}$ the
  \emph{external} boundary normal derivative applied to function on
  $B(x,R)\setminus\overline D$. In conclusion
  \eqref{usDistributionOutside} becomes
  \begin{align}
    0 &= \int_{S(x,R)} \big( \Phi(x,y)\partial_\nu u_s(y) - u_s(y)
    \partial_\nu\Phi(x,y) \big) ds(y) \notag \\ &\qquad +
    \int_{\partial D} \big( \Phi(x,y)(-\partial_\nu^{D_c}) u_s(y) -
    u_s(y) (-\partial_\nu^{D_c})\Phi(x,y) \big) ds(y). \label{usDout}
  \end{align}
  Finally, by interior elliptic regularity we see that $u_s$ is
  continuous (in fact real analytic) in some neighbourhood of
  $x$. Also, because $x$ is outside of $\partial D$ and $S(x,R)$, the
  individual boundary integrals above are continuous. Hence the
  equality in the sense of distributions is in fact a pointwise
  equality for continuous functions. In other words, both of
  \eqref{usDin} and \eqref{usDout} hold as continuous functions. We
  still remind that the integrals involving $\partial_\nu u_s$
  represent distribution pairings for an element of $H^{-1/2}(\partial
  D)$ with that of a smooth $\Phi$.

  Let us add \eqref{usDin} and \eqref{usDout}. By smoothness,
  $\partial_\nu^D \Phi = \partial_\nu^{D_c} \Phi$. Note that two-sided
  Sobolev traces of $H^1$-functions yield identical resuts, so the
  integrals of $u_s\partial_\nu^D\Phi$ and $u_s\partial_\nu^{D_c}\Phi$
  in \eqref{usDin} and \eqref{usDout} cancel out. The sum then gives
  \begin{align}
    u_s(x) &= \int_{S(x,R)} \big( \Phi(x,y)\partial_\nu u_s(y) -
    u_s(y) \partial_\nu\Phi(x,y) \big) ds(y) \notag \\ &\qquad +
    \int_{\partial D} \Phi(x,y) \big( \partial_\nu^D u_s -
    \partial_\nu^{D_c} u_s\big)(y) ds(y). \label{usDboth}
  \end{align}
  Note that as $R \rightarrow \infty$ the first integral in
  \eqref{usDboth} vanishes because $u_s$ satisfies the Sommerfeld
  radiation condition. Also, $u_s$ is $C^1$ outside of
  $\overline\Omega$ by elliptic interior regularity, so the second
  integral's integrand is zero when $y\notin\overline\Omega$. Thus,
  letting $R\to\infty$ gives
  \[
  u_s (x) = \int_{\Omega} \Phi(x,y) \big(\partial_\nu^D u_s -
  \partial_\nu^{D_c} u_s)(y) dy
  \]
  which implies the claim as $\partial_\nu^D u_s = \partial_3 u_s^+$
  and $\partial_\nu^{D_c} u_s = \partial_3 u_s^-$ on $\Omega
  \subset\R^2\times\{0\}$. Furthermore, as above, since $u_s$ is $C^1$
  outside of $\overline\Omega$, we see that $\partial_3 u_s^+ -
  \partial_3 u_s^- = 0$ outside of $\overline\Omega$, so the integrand
  in the statement is in $\widetilde H^{-1/2}(\Omega_0)$, as claimed.
\end{proof}

\bigskip
With the proposition above, we are almost ready to prove the formula
for the far-field of a wave scattered by a screen,
Theorem~\ref{usFarField}. But first let us prove a lemma.

\begin{lemma} \label{C1convergence}
  Let $k\in\R_+$ and $K\subset\R^3$ be a nonempty compact set. Then
  \[
  \lim_{r\to\infty} \sup_{\abs{x}=r} \sup_{y\in K} \abs{x} \abs{
    \partial_y^\alpha \left( \frac{e^{ik\abs{x-y}}}{\abs{x-y}} -
    \frac{e^{ik\abs{x}}}{\abs{x}} e^{-ik\hat x\cdot y} \right) } = 0
  \]
  for any multi-index $\alpha\in\N^3$ with $\abs{\alpha}\leq
  1$. Recall that $\hat x = x/\abs{x}$.
\end{lemma}
\begin{proof}
  The case of $\abs{\alpha}=0$ is well known, see for example the
  proof of Theorem~2.5 in \cite{CK}. For $\abs{\alpha}=1$ we will
  instead show the equivalent statement with $\partial_y^\alpha$
  replaced by $\nabla_y$. Recall the following differentiation rules
  \begin{itemize}
  \item $\nabla_y \abs{x-y}^s = -s \frac{x-y}{\abs{x-y}}
    \abs{x-y}^{s-1}$ for all $s\in\R$,
  \item $\nabla_y e^{ik\abs{x-y}} = -i k \frac{x-y}{\abs{x-y}}
    e^{ik\abs{x-y}}$, and
  \item $\nabla_y e^{-ik\hat x\cdot y} = -ik\hat x e^{-ik\hat x \cdot
    y}$.
  \end{itemize}
  These imply that
  \begin{align*}
    &\nabla_y \left( \frac{e^{ik\abs{x-y}}}{\abs{x-y}} -
    \frac{e^{ik\abs{x}}}{\abs{x}} e^{-ik\hat x\cdot y} \right)
    \\ &\qquad = -ik \frac{x-y}{\abs{x-y}}
    \frac{e^{ik\abs{x-y}}}{\abs{x-y}} + \frac{x-y}{\abs{x-y}}
    \frac{e^{ik\abs{x-y}}}{\abs{x-y}^2} + ik\hat x
    \frac{e^{ik\abs{x}}}{\abs{x}} e^{-ik\hat x \cdot y} \\ &\qquad =
    -ik \left( \frac{x-y}{\abs{x-y}} - \hat x \right)
    \frac{e^{ik\abs{x-y}}}{\abs{x-y}} - ik\hat x \left(
    \frac{e^{ik\abs{x-y}}}{\abs{x-y}} - \frac{e^{ik\abs{x}}}{\abs{x}}
    e^{-ik\hat x \cdot y} \right) \\ &\qquad \quad +
    \frac{x-y}{\abs{x-y}} \frac{e^{ik\abs{x-y}}}{\abs{x-y}^2}.
  \end{align*}
  Let us consider the three types of terms above. To prove the
  estimate, let us take the absolute value and multiply by
  $\abs{x}$. The last one gives
  \[
  \abs{x} \abs{\frac{x-y}{\abs{x-y}}
    \frac{e^{ik\abs{x-y}}}{\abs{x-y}^2} } =
  \frac{\abs{x}}{\abs{x-y}^2} \to 0
  \]
  uniformly as $y\in K$, $\abs{x}=r$ and $r\to\infty$. The first term
  gives
  \[
  \abs{x} \abs{ -ik\left( \frac{x-y}{\abs{x-y}} - \hat x\right)
    \frac{e^{ik\abs{x-y}}}{\abs{x-y}} } = k \frac{\abs{x}}{\abs{x-y}}
  \abs{ \frac{x-y}{\abs{x-y}} - \frac{x}{\abs{x}} }
  \]
  where can still estimate
  \[
  \abs{\frac{x-y}{\abs{x-y}} - \frac{x}{\abs{x}}} =
  \abs{\frac{x-y}{\abs{x-y}} \frac{\abs{x}-\abs{x-y}}{\abs{x}} -
    \frac{y}{\abs{x}}} \leq \frac{\abs{\abs{x}-\abs{x-y}}}{\abs{x}} +
  \frac{\abs{y}}{\abs{x}} \leq 2 \frac{\abs{y}}{\abs{x}}
  \]
  because $\abs{\abs{x}-\abs{x-y}} \leq \abs{y}$ by the triangle
  inequality. Thus the first term also tends to zero uniformly as
  $r\to\infty$. Lastly, the second one is estimated as
  \[
  \abs{x}\abs{-ik\hat x \left( \frac{e^{ik\abs{x-y}}}{\abs{x-y}} -
    \frac{e^{ik\abs{x}}}{\abs{x}} e^{-ik\hat x\cdot y} \right) } = k
  \abs{x} \abs{\frac{e^{ik\abs{x-y}}}{\abs{x-y}} -
    \frac{e^{ik\abs{x}}}{\abs{x}} e^{-ik\hat x\cdot y} }
  \]
  which tends to zero uniformly because this is the case
  $\abs{\alpha}=0$ proven at the beginning of this proof.
\end{proof}

\begin{proof}[Proof of Theorem~\ref{usFarField}]
  By the definition of the far-field there is a finite constant $C>0$
  independent of $x$ such that
  \[
  \abs{ u_\infty(\hat x) - \abs{x}e^{-ik\abs{x}} u_s(x)} \leq
  \frac{C}{\abs{x}}
  \]
  when $\abs{x}\to\infty$. Let us denote $\rho(y') = (\partial_3 u_s^+
  - \partial_3 u_s^-)(y^0)$. Then \eqref{usRepresentation} gives
  \[
  u_s^\infty(\hat x) = \lim_{\abs{x}\to\infty} \abs{x}e^{-ik\abs{x}}
  \big\langle \rho(y'), \Phi(x,y^0) \big\rangle_{y'}
  \]
  should the limit exist. The distribution pairing is over
  $y'\in\R^2$. We can rewrite
  \begin{align*}
    &\abs{x}e^{-ik\abs{x}} \big\langle \rho(y'), \Phi(x,y^0)
    \big\rangle \\ &\qquad = \left\langle \rho(y'), \abs{x}
    e^{-ik\abs{x}} \Phi(x,y^0) - e^{-ik\hat x\cdot
      y^0}/(4\pi)\right\rangle_{y'} \\ &\qquad \quad + \frac{1}{4\pi}
    \big\langle \rho(y'), e^{-ik\hat x\cdot y^0} \big\rangle_{y'}.
  \end{align*}
  We can write the $C^1$-test function on the second line as
  \begin{align*}
    &\abs{x} e^{-ik\abs{x}} \Phi(x,y^0) - e^{-ik\hat x\cdot
      y^0}/(4\pi) \\ &\qquad = \frac{e^{-ik\abs{x}} \abs{x}}{4\pi}
    \left( \frac{e^{ik\abs{x-y^0}}}{\abs{x-y^0}} -
    \frac{e^{ik\abs{x}}}{\abs{x}} e^{-ik\hat x\cdot y^0} \right)
  \end{align*}
  which convergest to zero in the $C^1$ topology over $y'$, and a
  fortiori $y^0$, restricted to any compact set by
  Lemma~\ref{C1convergence}. Note that the $C^1$-seminorms are taken
  with respect to the $y'$-variable, and the absolute value makes the
  $e^{-ik\abs{x}}$ that doesn't appear in the lemma disappear. Hence
  the application of the lemma is allowed. Elements of $\widetilde
  H^{-1/2}(\Omega_0)$ act well on $C^1$-functions, so the distribution
  pairing with $\rho$ and the test function tends to zero. Thus
  \[
  \lim_{\abs{x}\to\infty} \abs{x}e^{-ik\abs{x}} \big\langle \rho(y'),
  \Phi(x,y^0) \big\rangle_{y'} = \frac{1}{4\pi} \left\langle \rho(y'),
  e^{-ik\hat x\cdot y^0} \right\rangle_{y'}
  \]
  as claimed.
\end{proof}

\section{Solving the inverse problem} \label{sect:inverse}
We are ready to tackle the inverse problem in this section.

\begin{lemma} \label{rhoDetermination}
  Let $k\in\R_+$ and $\rho\in\mathscr E'(\R^2)$ be a distribution of
  compact support. Let\footnote{If $\rho$ is integrable then
    $u^\infty_s(\hat x) = \frac{1}{4\pi} \int_{\R^2} e^{-ik\hat x\cdot
      y^0} \rho(y') dy'$.}
  \begin{equation} \label{generalRho2farField}
    u^\infty_s(\hat x) = \frac{1}{4\pi} \left\langle \rho, e^{-ik\hat
      x\cdot y^0} \right\rangle
  \end{equation}
  for $\hat x\in\mathbb S^2$ and where the distribution pairing is
  over the variable $y'=(y_1,y_2)\in\R^2$. Then $\rho$ is uniquely
  determined by $u^\infty_s$.
\end{lemma}
\begin{proof}
  The operator mapping $\rho \mapsto u^\infty_s$ is bounded and linear
  $\mathcal E'(\R^2) \to C^0(\mathbb S^2)$. This is because $\hat x
  \mapsto \big( y' \mapsto \exp(-ik\hat x\cdot y^0)\big)$ is
  continuous $\mathbb S^2 \to \mathcal E(\R^2)$. So it is enough to
  show that $\rho=0$ if $u^\infty_s=0$. Let us assume the latter. For
  $\xi'\in\R^2$ we have
  \[
  \hat \rho(\xi') = \frac{1}{2\pi} \left\langle \rho, 
  e^{-i\xi'\cdot y'} \right\rangle
  \]
  where the distribution pairing is over the variable
  $y'\in\R^2$. This looks similar to the formula
  \eqref{generalRho2farField} in the statement. We can rewrite
  \[
  {-ik\hat x\cdot y^0} = {-i k(\hat x_1, \hat x_2, \hat x_3)\cdot
    (y_1,y_2,0)} = {- i (k\hat x_1,k\hat x_2) \cdot (y_1,y_2)}.
  \]
  Thus
  \begin{equation}
    u^\infty_s(\hat x) = \frac{1}{2} \hat \rho(k\hat x_1, k\hat x_2).
  \end{equation}
  The left-hand side is zero for all $\hat x\in\mathbb S^2$. When
  $\hat x$ goes through the whole of $\mathbb S^2$, the sum including
  only two of the squares, $\hat x_1^2 + \hat x_2^2$, goes through the
  whole interval $(0,1)$. Alternatively
  \[
  \hat \rho(\xi') = 2 u^\infty_s \left(\xi_1/k, \xi_2/k, \sqrt{k^2 -
    \xi_1^2 + \xi_2^2}/k \right) = 0
  \]
  for all $\abs{\xi'} \leq k$. Since $\rho$ has compact support, $\hat
  \rho$ can be extended to an entire function on $\C^2$. Since it
  vanishes on an open subset of $\R^2$ it must be the zero
  function. Hence $u^\infty_s = 0$ implies $\rho=0$.
\end{proof}

\begin{lemma} \label{OmegaDetermination}
  Let $(\Delta+k^2)u_i=0$ in $\R^3$. Let $\Omega\subset\R^3$ be a
  screen and $u_s$ satisfy the direct scattering problem
  \ref{directScat}. Denote
  \[
  \rho(x') = \partial_3 u_s^+(x^0) - \partial_3 u_s^-(x^0)
  \]
  for $x'\in\R^2$ and its properties are given in
  Proposition~\ref{prop:usRep}. If $u_i(x',x_3) \neq -u_i(x',-x_3)$
  for some $x\in\R^3$ then
  \begin{equation}
    \overline{\Omega_0} = \operatorname{supp} \rho
  \end{equation}
  for the shape $\Omega_0$ of the screen $\Omega$.
\end{lemma}
\begin{proof}
  The function $\rho$ is a well-defined $H^{-1/2}(\Omega_0)$-function
  by Proposition~\ref{prop:usRep} so in particular
  $\operatorname{supp}\rho \subset\overline{\Omega_0}$. It remains to
  prove that $\overline{\Omega_0} \subset \operatorname{supp} \rho$.

  Assume the contrary, that $\overline{\Omega_0}$ is not contained in
  the support of $\rho$. Then neither is $\Omega_0$ because if
  $\Omega_0\subset\operatorname{supp}\rho$ then $ \overline{\Omega_0}
  \subset \overline{\operatorname{supp}\rho} =
  \operatorname{supp}\rho$. Because $\Omega_0$ is an open set and
  $\operatorname{supp}\rho$ is closed there is $x'_0\in\Omega_0$ and
  $r>0$ such that $B(x'_0,r) \subset \Omega_0 \setminus
  \operatorname{supp}\rho$.

  Let us study the behaviour of $u_s$ in the tube $B(x'_0,r) \times
  \R$. We have $\rho = 0$ on $B(x'_0,r)$. Recall formula
  \eqref{usRho}, which combined with the vanishing of $\rho$ implies
  that $(\Delta+k^2)u_s=0$ in the whole tube, and interior elliptic
  regularity implies that $u_s$ is smooth there. In addition the
  formula implies that $u_s(x_1,x_2,x_3) = u_s(x_1,x_2,-x_3)$ for all
  $x$ in the tube. The vanishing of $\rho$ gives $\partial_3 u_s^+ =
  \partial_3 u_s^-$ on the base of the tube. These two imply that
  actually $\partial_3 u_s(x',0) = 0$ for $x'\in B(x'_0,r)$.

  We have the following
  \begin{align}
    u_s &= -u_i, \label{us0}\\
    \partial_3 u_s &= 0 \label{us1}
  \end{align}
  on $B(x'_0,r)\times\{0\}$. Let us calculate the higher order
  derivatives. Note that $\partial_3^j$ and $(\Delta+k^2)$ commute,
  and $(\Delta+k^2)u_s = 0$ in the tube. Thus
  \[
  0 = \partial_3^j(\Delta+k^2)u_s = (\Delta+k^2)\partial_3^j u_s =
  (\Delta'+k^2)\partial_3^j u_s + \partial_3^{j+2} u_s
  \]
  in the tube, and we denote $\Delta' = \partial_1^2 +
  \partial_2^2$. This gives $\partial_3^{j+2} u_s = -
  (\Delta'+k^2)\partial_3^j u_s$. Let us restrict ourselves to
  $B(x'_0,r)\times\{0\}$ next. By induction and
  \eqref{us0}--\eqref{us1} we see that
  \[
  \partial_3^j u_s = \begin{cases}
    (-1)^{j+1} (\Delta'+k^2)^j u_i, &j\in2\N,\\
    0, &j\in2\N+1
  \end{cases}
  \]
  on $B(x'_0,r)\times\{0\}$. This can still be simplified! Recall that
  $u_i$ is an incident wave, so $(\Delta+k^2)u_i=0$ everywhere. This
  means that $(\Delta'+k^2) u_i = -\partial_3^2 u_i$, and a fortiori
  $(\Delta'+k^2)^j u_j = (-\partial_3^2)^j u_i$ everywhere by the
  commutating of $\partial_3^2$ and $(\Delta'+k^2)$. This implies
  \begin{equation}
    \partial_3^j u_s = \begin{cases}
      - \partial_3^j u_i, &j\in2\N,\\
      0, &j\in2\N+1.
    \end{cases}
  \end{equation}
  The other derivatives, $\partial_1$ and $\partial_2$ commute with
  each other and $\partial_3$, so finally we have
  \begin{equation} \label{usAlpha}
    \partial^\alpha u_s = \begin{cases}
      - \partial^\alpha u_i, &\alpha_3\in2\N,\\
      0, &\alpha_3\in2\N+1
    \end{cases}
  \end{equation}
  on $B(x'_0,r)\times\{0\}$ for all multi-indices $\alpha\in\N^3$.

  Let us define
  \[
  \tilde u_i(x) = \frac{1}{2} \big( u_i(x_1,x_2,x_3) +
  u_i(x_1,x_2,-x_3)\big)
  \]
  for all $x\in\R^3$. This satisfies the Helmholtz equation
  everywhere, and is an incident wave because $u_i$ is one. We see
  that
  \[
  \partial^\alpha \tilde u_i(x) = \frac12 \big( \partial^\alpha
  u_i(x_1,x_2,x_3) + (-1)^{\alpha_3} \partial^\alpha u_i(x_1,x_2,-x_3)
  \big)
  \]
  so
  \begin{equation} \label{uiTildeAlpha}
    \partial^\alpha \tilde u_i = \begin{cases} \partial^\alpha u_i,
      &\alpha_3\in2\N,\\ 0, &\alpha_3\in2\N+1
    \end{cases}
  \end{equation}
  on $B(x'_0,r)\times\{0\}$. By \eqref{usAlpha} we see immediately
  that $\partial^\alpha u_s = - \partial^\alpha \tilde u_i$ on the
  base of the tube for all $\alpha\in\N^3$. Both functions $u_s$ and
  $-\tilde u_i$ satisfy the Helmholtz equation not only in the tube
  but also in $\R^3 \setminus \overline B(0,R)$, where $R>0$ is large
  enough that $\overline\Omega \subset B(0,R)$. Solutions of the
  Helmholtz equation are real-analytic. Because their
  Taylor-expansions at $(x'_0,0)$ are equal, the functions are equal
  in the component of $\big(B(x'_0,r)\times\R\big) \cup
  \big(\R^3\setminus\overline B(0,R)\big)$ that contains $(x'_0,0)$,
  so in particular $u_s= -\tilde u_i$ in all of $\R^3\setminus
  \overline B(0,R)$.

  The function $u_s$ satisfies the Sommerfeld radiation condition, so
  so does $\tilde u_i$. On the other hand $(\Delta+k^2)\tilde u_i=0$
  in all of $\R^3$, so $\tilde u_i$ is the zero function\footnote{Use
    e.g. \eqref{greenRepresentation} for a larbe ball whose radius
    grows to infinity. The boundary integral decreases to zero as was
    seen for the first integral in \eqref{usDboth}.}, which means that
  $u_i$ is antisymmetric with respect to $\R^2\times\{0\}$, a
  contradiction. Hence $\overline{\Omega_0} \subset
  \operatorname{supp}\rho$.
\end{proof}

\bigskip
The solution to the inverse problem of determining a screen $\Omega$
from the knowledge of a single incident wave $u_i$ and the
corresponding far-field $u_s^\infty$ scattered from the screen comes
from a combination of determining $\rho$ from the far-field, and then
$\Omega$ from $\rho$. There is a slight suprise, namely that the
problem is only solvable for incident waves that are not too
(anti)symmetric. However, one sees that antisymmetry is not the
deciding factor: what matters is whether $u_i$ is identically zero on
the screen. By a similar argument as that at the end of the proof of
Lemma~\ref{OmegaDetermination}, we see that if $u_i=0$ on a non-empty
open subset of $\R^2\times\{0\}$ then $u_i(x',x_3) = -u_i(x',-x_3)$
for all $x\in\R^3$. It is interesting to see that partial invisibility
is achieved inside thickened screens as long as the incident plane
wave comes from a direction almost parallel to the screen's normal
\cite{DLU}. The direction of incident waves seems very important in
scattering from objects that are thin in one direction.
\begin{proof}[Proof of Theorem~\ref{inverseSolution}]
  Theorem~\ref{usFarField} and Lemma~\ref{rhoDetermination} imply that
  $\rho=\tilde\rho$ when $u_s^\infty = \tilde u_s^\infty$. If $u_i$ is
  not antisymmetric with respect to $\R^2\times\{0\}$ then
  \[
  \overline{\Omega_0} = \operatorname{supp} \rho = \operatorname{supp}
  \tilde\rho = \overline{\tilde\Omega_0}
  \]
  by Lemma~\ref{OmegaDetermination}. Because $\Omega_0$ is a smooth
  domain, we have $\Omega_0=\operatorname{int}\overline{\Omega_0}$,
  and similarly for $\tilde\Omega_0$. Thus the equation above implies
  $\Omega_0=\tilde\Omega_0$ and by lifting, $\Omega=\tilde\Omega$.

  If $u_i$ is antisymmetric then $u_i=0$ everywhere on
  $\R^2\times\{0\}$ and $u_s=0$ satisfies all conditions of the direct
  scattering problem. Since solutions to the direct scattering problem
  \eqref{directScat} are unique by \cite[Thms~2.5--2.7]{Stephan}, this
  is the only solution. Thus $u_s=\tilde u_s=0$ and the same holds for
  their far-fields. This is irrespective of the shape of
  $\Omega,\tilde\Omega\subset\R^2$.
\end{proof}

\section*{Acknowledgements}
The research of the authors is supported by Estonian Research Council
grant PRG832. We also thank Markku Lehtinen for useful discussions.

\end{document}